\documentclass[a4paper,oneside,10pt]{article}%
\usepackage{amsmath}
\usepackage{amsfonts}
\usepackage{amssymb}
\usepackage{graphicx}
\usepackage[square,numbers,sort&compress]{natbib}%
\providecommand{\U}[1]{\protect\rule{.1in}{.1in}}

\newtheorem{theorem}{Theorem}[section]

\newtheorem{corollary}[theorem]{Corollary}

\newtheorem{definition}[theorem]{Definition}

\newtheorem{example}[theorem]{Example}

\newtheorem{lemma}[theorem]{Lemma}

\newtheorem{problem}[theorem]{Problem}
\newtheorem{proposition}[theorem]{Proposition}
\newtheorem{remark}[theorem]{Remark}

\newenvironment{proof}[1][Proof]{\noindent\textbf{#1.} }{\ \rule{0.5em}{0.5em}}
\numberwithin{equation}{section}

\begin{document}

\title{A generalized Neyman-Pearson lemma for sublinear expectations}
\author{Chuanfeng Sun\thanks{School of Mathematical Sciences, University of Jinan,
Jinan, Shandong 250022, P.R. China. (sms\_suncf@ujn.edu.cn).}
\and Shaolin Ji\thanks{Institute for Financial Studies and Institute of
Mathematics, Shandong University, Jinan, Shandong 250100, P.R. China
(Jsl@sdu.edu.cn, Fax: +86 0531 88564100).}}
\maketitle

\textbf{Abstract}. In this paper, the Neyman-Pearson lemma for general
sublinear expectations is studied. We weaken the assumptions for sublinear
expectations in \cite{CK} and give a completely new method to study this
problem. Applying Mazur-Orlicz Theorem and the decomposition theorem of
finitely additive set functions, we prove that the optimal test still has the
reminiscent form as in the classical Neyman-Pearson lemma. Finally, for the
special sublinear expectation which can be represented by a family of
probability measures, we give a sufficient condition for the existence of the
optimal test and show the form of the optimal test selected in $L_{c}^{1}%
$-space which is introduced by Peng \cite{Peng} in his nonlinear-expectation framework.


\section{Introduction}

The Neyman-Pearson lemma is a basic result of hypothesis testing in
statistics. In more details, when we want to discriminate between two
probability measures $P$ and $Q$ on a given measurable space $(\Omega
,\mathcal{F})$, we can choose a randomized test $X:\Omega\rightarrow
\lbrack0,1]$ which rejects $P$ on $\omega$ with probability $X(\omega)$. Then,
$E_{P}[X]$ is the probability of rejecting $P$ when it is true (Type I error)
and $E_{Q}[1-X]$ is the probability of accepting $P$ when it is false (Type II
error). An optimal test minimizes Type II error while keeping Type I error
below a given acceptable significance level $\alpha\in(0,1)$. The
Neyman-Pearson lemma tells us the form that the optimal test should satisfy
(see \cite{CK} or \cite{Shao}).

It is a natural challenging problem to extend this lemma to simple hypothesis
testing for nonlinear probabilities and expectations (or risk measures). Huber
and Strassen \cite{HS} studied hypothesis testing problem for Choquet
capacities. Cvitanic and Karatzas \cite{CK} studied the min-max test by using
convex duality method. Schied gave a Neyman-Pearson lemma for law-invariant
coherent risk measures and robust utility functionals. Ji and Zhou \cite{JZ}
studied hypothesis tests for $g$-probabilities. Rudloff and Karatzas \cite{RK}
studied composite hypothesis by using the Fenchel duality. The similar problem
also arises in the financial mathematics (refer \cite{FL-1,FL-2,FS-1} and
\cite{Rud}).

To study Neyman-Pearson lemma for nonlinear probabilities and expectations, we
are interesting whether there still exists a representative pair of
probabilities $(Q^{\ast},P^{\ast})$ such that the optimal test is just the
optimal test between the simple hypotheses $Q^{\ast}$ and $P^{\ast}$. If such
a representative pair of probabilities exists, then the optimal test has the
reminiscent form, just like the classical Neyman-Pearson lemma. In most
literatures, the convex duality method is employed to study the nonlinear
Neyman-Pearson lemma and the corresponding pair of simple hypotheses is found.

For the Neyman-Pearson lemma for sublinear expectations, Cvitani\'{c} and
Karatzas \cite{CK} studied the Neyman-Pearson lemma by the convex duality
method. In \cite{CK}, Cvitanic and Karatzas assumed that the two sublinear
expectations are generated by two families of probability measures
$\mathcal{P}$ and $\mathcal{Q}$, and there exists a probability measure $K$
such that $P\ll K$, $Q\ll K$, $\forall P\in\mathcal{P}$, $\forall
Q\in\mathcal{Q}$. Then, under the closed assumption of the set of densities
which generate the sublinear expectation, they proved that the optimal test
has the reminiscent form.

However, for sublinear expectations, as shown in Example \ref{exa5.2}, there
may be no reference probability measure $K$, so that all elements in
$\mathcal{P}$ and $\mathcal{Q}$ are dominated by it. In this paper, our goal
is to study the Neyman-Pearson lemma for general sublinear expectations, and
remove the assumptions that there exists a reference probability measure $K$
and the set of generated densities is closed in \cite{CK}. In this case, the
convex dual method in \cite{CK} is no longer applicable and we give a
completely new method to study this problem.

Applying Mazur-Orlicz theorem, our original problem can be transformed into a
simple hypothesis testing problem for finitely additive set functions. With
the help of Theorem 1.23 in \cite{YH}, the finitely additive set function $Q$
can be uniquely decomposed as
\[
Q=\lambda Q^{c}+(1-\lambda)Q^{p},
\]
where $Q^{c}$ is a probability measure, $Q^{p}$ is a pure additive set
function and $\lambda\in\lbrack0,1]$. Based on this result and by introducing
two reasonable assumptions, we have obtained a necessary condition that the
optimal test satisfies. We prove that the optimal test still has the
reminiscent form as in the classical Neyman-Pearson lemma. Then, for the
special sublinear expectation which can be represented by a family of
probability measures, we give a sufficient condition for the existence of the
optimal test. Finally, we study the hypothesis test when the test function is
selected in $L_{c}^{1}$-space which is introduced by Peng \cite{Peng} in his
nonlinear-expectation framework. The form of the optimal test is obtained and
an example is given to illustrate an optimal test for $G$-expectation in
\cite{Peng}.

This paper is organized as follows. In section 2, we formulate our problem. In
section 3, we study the properties of the optimal tests and obtain the
representation of the optimal tests. Some results are discussed when the
sublinear expectation can be represented by a family of probability measures.

\section{{Problem Formulation }\newline}

Let $(\Omega,\mathcal{F})$ be a measurable space and $\mathcal{X}$ be the set
of all bounded measurable functions on it. Then $\mathcal{X}$ is a Banach
space endowed with the supremum norm. Let $\mathcal{X}^{\ast}$ be the dual
space of $\mathcal{X}$. Denote by $\mathbb{N}$ the set of natural numbers and
$\mathbb{R}$ the set of real numbers.

\begin{definition}
\label{def0} A mapping $\rho:\mathcal{X\longrightarrow}\mathbb{R}$ is called a
sublinear expectation if for any $\xi_{1},\xi_{2}\in\mathcal{X}$, we have

(i) Monotonicity: $\rho(\xi_{1})\geq\rho(\xi_{2})$ if $\xi_{1}\geq\xi_{2}$.

(ii) Constant preserving: $\rho(c)=c$ for $c\in\mathbb{R}$.

(iii) Sub-additivity: $\rho(\xi_{1}+\xi_{2})\leq\rho(\xi_{1})+\rho(\xi_{2})$
for each $\xi_{1},\xi_{2}\in\mathcal{X}$.

(iv) Positive homogeneity: $\rho(\lambda\xi)=\lambda\rho(\xi)$ for any real
number $\lambda\geq0$.
\end{definition}

Denote the conjugation operator of $\rho$ by $\bar{\rho}$, i.e., for any
$X\in\mathcal{X}$,%
\[
\bar{\rho}(X)=-\rho(-X).
\]
By Theorem A.50 in \cite{FS}, for any linear operator $L\in\mathcal{X}^{\ast}%
$, there exists a unique bounded finitely additive set function $P$ on the
measurable space $(\Omega,\mathcal{F})$ such that
\[
L(X)=\int XdP\quad\text{for all}\quad X\in\mathcal{X}.
\]
In order to show the one-to-one correspondence between the element in
$\mathcal{X}^{\ast}$ and its corresponding bounded finitely additive set
function, we denote the linear operator $L\in\mathcal{X}^{\ast}$ by $E_{P}$.
For a given sublinear expectation $\rho$, set
\[
\mathcal{P}:=\{P\mid E_{P}\leq\rho\}
\]
where $P$ is a bounded finitely additive set function $P$ on $(\Omega
,\mathcal{F})$. Then, by Proposition 2.85 in \cite{FS}, we have
\[
\rho(X)=\sup_{P\in\mathcal{P}}E_{P}[X]\quad\text{and}\quad\bar{\rho}%
(X)=\inf_{P\in\mathcal{P}}E_{P}[X].
\]

In this paper, we study the Neyman-Pearson fundamental lemma for two sublinear
expectations $\rho_{1}$ and $\rho_{2}$. In more details, we want to
discriminate between $\rho_{1}$ and $\rho_{2}$, by selecting a randomized test
function $X:\Omega\rightarrow\lbrack0,1]$, under a given significance level
$\alpha\in(0,1)$. In this framework, $\rho_{1}(X)$ measures the expectation of
rejecting $\rho_{1}$ when it is true (Type I error) and $\rho_{2}(1-X)$
measures the expectation of accepting $\rho_{1}$ when it is false (Type II
error). It is well-known that it is generally impossible to minimize both
types of errors simultaneously. So our aim is to choose an optimal test
function $X\in\lbrack0,1]$ which minimizes the Type II error $\rho_{2}(1-X)$
and makes the Type I error $\rho_{1}(X)$ less than the given significance
level $\alpha$.

Set
\[
\mathcal{X}_{\alpha}=\{X\mid\rho_{1}(X)\leq\alpha,0\leq X\leq1,X\in
\mathcal{X}\}.
\]

Consider the following problem:

\begin{problem}
For a given significance level $\alpha\in(0,1)$,
\[
\text{minimize }\rho_{2}(1-X)
\]
over the set $\mathcal{X}_{\alpha}$.
\end{problem}

It is easy to verify that the above problem is equivalent to the following one:

\begin{problem}
\label{problem2.1} For a given significance level $\alpha\in(0,1)$,
\begin{equation}
\text{Maximize }\bar{\rho}_{2}(X) \label{2.1}%
\end{equation}
over the set $\mathcal{X}_{\alpha}$.
\end{problem}

\begin{definition}
We call $X^{\ast}$ an optimal test of Problem \ref{problem2.1} if $X^{\ast}%
\in\mathcal{X}_{\alpha}$ and $\bar{\rho}_{2}(X^{\ast})=\underset{X\in
\mathcal{X}_{\alpha}}{\max}\bar{\rho}_{2}(X).$
\end{definition}

In the following, we will use the set $\mathcal{P}$ (resp. $\mathcal{Q}$) to
denote the set of additive set functions dominated by $\rho_{1}$ (resp.
$\rho_{2}$), i.e.,
\begin{align*}
\mathcal{P}  &  =\{P\mid E_{P}\leq\rho_{1}\};\\
\mathcal{Q}  &  =\{Q\mid E_{Q}\leq\rho_{2}\}.
\end{align*}

It is obvious that
\begin{align*}
\rho_{1}(X)  &  =\sup_{P\in\mathcal{P}}E_{P}[X]\quad\text{and}\quad\rho
_{2}(X)=\sup_{Q\in\mathcal{Q}}E_{Q}[X];\\
\bar{\rho}_{1}(X)  &  =\inf_{P\in\mathcal{P}}E_{P}[X]\quad\text{and}\quad
\bar{\rho}_{2}(X)=\inf_{Q\in\mathcal{Q}}E_{Q}[X].
\end{align*}

Note that the elements in $\mathcal{P}$ and $\mathcal{Q}$ are only finitely
additive, not necessarily countably additive.

Now we give an example to show that there may be no reference probability
measure $K$ such that all elements in $\mathcal{P}$ and $\mathcal{Q}$ are
dominated by it.

\begin{example}
\label{exa5.2} Let $\Omega=[0,1]$, $\mathcal{F}$ be all the Borel set on
$[0,1]$ and
\[
\delta_{x}(\omega)=\left\{
\begin{array}
[c]{l@{}c}%
\frac{2}{3},\quad & \quad\omega=x;\\
\frac{1}{3},\quad & \quad\omega=x+\frac{1}{2};\\
0,\quad & \quad\text{otherwise},
\end{array}
\right.  \;Q(\omega)=\left\{
\begin{array}
[c]{l@{}c}%
\frac{1}{2^{k+1}},\quad & \quad\omega=\frac{1}{2^{k}},k\geq1;\\
\frac{1}{2},\quad & \quad\omega=\frac{3}{4};\\
0,\quad & \quad\text{otherwise},
\end{array}
\right.  \;X^{\ast}(\omega)=\left\{
\begin{array}
[c]{l@{}c}%
\frac{1}{2},\quad & \quad\omega=\frac{1}{2^{k}},k\geq1,k\not =2;\\
0,\quad & \quad\omega=\frac{1}{4};\\
1,\quad & \quad\omega=\frac{3}{4};\\
0,\quad & \quad\text{otherwise}.
\end{array}
\right.
\]
Note that $\delta_{x}$ ($0\leq x\leq\frac{1}{2}$) is a measure on $[0,1]$. Let
$\mathcal{P}=\{\delta_{x},x\in\lbrack0,\frac{1}{2}]\}$ and $\mathcal{Q}%
=\{Q\}$. Then, there does not exist a $K$ such that for any $P\in\mathcal{P}$,
$P\ll K$. If $\alpha=\frac{1}{3}$, the optimal test exists but not unique, and
$X^{\ast}$ is an optimal test.
\end{example}

\section{Properties and representation of optimal tests\label{sec3}}

In this section, we will always assume the optimal test of Problem
\ref{problem2.1} exists.

\subsection{Properties of optimal tests\label{property}}

We first study the properties of optimal tests.

\begin{lemma}
\label{cor4.2} There exists a finitely additive set function $Q\in\mathcal{Q}$
such that
\[
\sup_{X\in\mathcal{X}_{\alpha}}E_{Q}[X]=\sup_{X\in\mathcal{X}_{\alpha}}%
\bar{\rho}_{2}(X)
\]
and $Q(\Omega)=1$.
\end{lemma}

\begin{proof}
Let $\bar{X}:=-X$, $\bar{\mathcal{X}}_{\alpha}:=\{\bar{X}:X\in\mathcal{X}%
_{\alpha}\}$. Then, $\bar{\mathcal{X}}_{\alpha}$ is a convex set. By
Mazur-Orlicz theorem (see Lemma 1.6 of chapter I in \cite{Simons}), there
exists a finitely additive set function $Q\in\mathcal{Q}$ such that
\[
\inf_{\bar{X}\in\bar{\mathcal{X}}_{\alpha}}E_{Q}[\bar{X}]=\inf_{\bar{X}\in
\bar{\mathcal{X}}_{\alpha}}\rho_{2}(\bar{X}).
\]
Thus,
\[
\sup_{X\in\mathcal{X}_{\alpha}}E_{Q}[X]=\sup_{X\in\mathcal{X}_{\alpha}}%
\bar{\rho}_{2}(X).
\]

On the other hand, since
\[
1=-\rho_{2}[-1]=\bar{\rho}_{2}[1]\leq E_{Q}[1]\leq\rho_{2}[1]=1,
\]
we obtain $Q(\Omega)=E_{Q}[1]=1$.
\end{proof}

\begin{remark}
\label{rem4.3} Similarly, if we consider the problem $\inf\limits_{_{X\in
\mathcal{D}}}\rho_{1}(X)$ for some convex set $\mathcal{D}$ of $\mathcal{X}$,
there exists $P\in\mathcal{P}$ such that
\[
\inf_{X\in\mathcal{D}}E_{P}[X]=\inf_{X\in\mathcal{D}}\rho_{1}(X).
\]

\end{remark}

Set%
\[
\mathcal{\bar{Q}=}\{Q\in\mathcal{Q}\mid\sup\limits_{X\in\mathcal{X}_{\alpha}%
}E_{Q}[X]=\sup\limits_{X\in\mathcal{X}_{\alpha}}\bar{\rho}_{2}(X)\}.
\]

\begin{proposition}
\label{lem4.2} If $X^{\ast}$ is an optimal test of Problem \ref{problem2.1},
then for any $Q\in\mathcal{\bar{Q}}$, we have
\[
E_{Q}[X^{\ast}]=\sup_{X\in\mathcal{X}_{\alpha}}E_{Q}[X].
\]

\end{proposition}

\begin{proof}
By Lemma \ref{cor4.2}, $\mathcal{\bar{Q}}$ is non-empty. For any
$Q\in\mathcal{\bar{Q}}$, since $\bar{\rho}_{2}[X^{\ast}]=\sup\limits_{X\in
\mathcal{X}_{\alpha}}\bar{\rho}_{2}[X]$ and
\[
\bar{\rho}_{2}[X^{\ast}]\leq E_{Q}[X^{\ast}]\leq\sup_{X\in\mathcal{X}_{\alpha
}}E_{Q}[X]=\sup_{X\in\mathcal{X}_{\alpha}}\bar{\rho}_{2}(X),
\]
then $E_{Q}[X^{\ast}]=\sup\limits_{X\in\mathcal{X}_{\alpha}}E_{Q}[X]$.
\end{proof}

\begin{lemma}
\label{H2} For any $N\in\mathbb{N}$ and $X\in\mathcal{X}_{\alpha}$ such that
$\rho_{1}(X)>0$, we have
\[
\rho_{1}(X)>\rho_{1}[(X-\frac{1}{N})^{+}].
\]

\end{lemma}

\begin{proof}
For any $X\in\mathcal{X}_{\alpha}$ such that $\rho_{1}(X)>0$ and
$N\in\mathbb{N}$, by Remark \ref{rem4.3}, there exists $P_{N}\in\mathcal{P}$
such that
\[
E_{P_{N}}[(X-\frac{1}{N})^{+}]=\rho_{1}[(X-\frac{1}{N})^{+}].
\]
Take $A:=\{\omega:X(\omega)\geq\frac{1}{N}\}$. If $P_{N}(A)=0$,
\[
\rho_{1}[(X-\frac{1}{N})^{+}]=E_{P_{N}}[(X-\frac{1}{N})^{+}]\leq E_{P_{N}%
}[I_{A}]=0.
\]
Since $\rho_{1}(X)>0$, then $\rho_{1}(X)>\rho_{1}[(X-\frac{1}{N})^{+}]$. If
$P_{N}(A)>0$,
\[%
\begin{array}
[c]{r@{}l}
& \rho_{1}(X)-\rho_{1}[(X-\frac{1}{N})^{+}]\geq E_{P_{N}}[X]-E_{P_{N}%
}[(X-\frac{1}{N})^{+}]\\
= & E_{P_{N}}[X-(X-\frac{1}{N})^{+}]\geq E_{P_{N}}[(X-(X-\frac{1}{N}%
)^{+})I_{A}]=\frac{P_{N}(A)}{N}>0.
\end{array}
\]
Thus, $\rho_{1}(X)>\rho_{1}[(X-\frac{1}{N})^{+}]$.
\end{proof}

The following definition of purely finitely additive set function comes from
\cite{YH}.

\begin{definition}
\label{def4.1} We call a finitely additive set function $Q$ is pure additive,
if $Q(\Omega)=1$ and there exists a sequence of sets $A_{n}\downarrow
\emptyset$ such that for any $n\in\mathbb{N}$, $Q(A_{n})=1$.
\end{definition}

By Theorem 1.23 in \cite{YH}, the finitely additive set function $Q$ can be
uniquely decomposed as
\[
Q=\lambda Q^{c}+(1-\lambda)Q^{p},
\]
where $Q^{c}$ is a probability measure, $Q^{p}$ is a pure additive set
function and $\lambda\in\lbrack0,1]$.

We need the following assumption:

(H1) For any $A_{n}\downarrow\emptyset$ such that $\lim\limits_{n\rightarrow
\infty}\rho_{2}(I_{A_{n}})\neq0$, we have $\lim\limits_{n\rightarrow\infty
}\rho_{1}(I_{A_{n}})=0$.

\begin{proposition}
\label{lem4.3} Under (H1), if $X^{\ast}$ is an optimal test of Problem
\ref{problem2.1}, then for any $Q\in\mathcal{\bar{Q}}$,
\[
E_{(1-\lambda)Q^{p}}[X^{\ast}]=1-\lambda.
\]

\end{proposition}

\begin{proof}
When $\lambda=1$, the result holds obviously. In the following, we assume
$\lambda<1$.

If $E_{(1-\lambda)Q^{p}}[X^{\ast}]=\lambda_{0}<1-\lambda$, then there exists a
large enough $N\in\mathbb{N}$ such that $\lambda_{0}+\frac{1}{N}<1-\lambda$.
Since $Q^{p}$ is a pure additive set function, there exists a sequence of sets
$A_{n}\downarrow\emptyset$ such that for any $n\in\mathbb{N}$, $Q^{p}%
(A_{n})=1$. When $\rho_{1}(X^{\ast})=0$, by (H1), there exists a set $A^{\ast
}\in\{A_{n};n\in\mathbb{N}\}$ such that
\[
E_{(1-\lambda)Q^{p}}[I_{A^{\ast}}]=1-\lambda\quad\text{and}\quad\rho
_{1}(I_{A^{\ast}})\leq\alpha.
\]
When $\rho_{1}(X^{\ast})>0$, by (H1) and Lemma \ref{H2}, there exists a set
$A^{\ast\ast}\in\{A_{n};n\in\mathbb{N}\}$ such that
\[
E_{(1-\lambda)Q^{p}}[I_{A^{\ast\ast}}]=1-\lambda
\]
and
\[
\rho_{1}(I_{A^{\ast\ast}})\leq\rho_{1}(X^{\ast})-\rho_{1}((X^{\ast}-\frac
{1}{N})^{+}).
\]

Take $A=A^{\ast}\cap A^{\ast\ast}$. Then, $A\in\{A_{n};n\in\mathbb{N}\}$ and
$E_{(1-\lambda)Q^{p}}[I_{A}]=1-\lambda$. Let
\[
X^{N}=(X^{\ast}-\frac{1}{N})^{+}I_{A^{c}}+I_{A}.
\]
When $\rho_{1}(X^{\ast})=0$, we have $\rho_{1}((X^{\ast}-\frac{1}{N})^{+})=0$
and
\[
\rho_{1}(X^{N})\leq\rho_{1}((X^{\ast}-\frac{1}{N})^{+}I_{A^{c}})+\rho
_{1}(I_{A})=\rho_{1}(I_{A})\leq\rho_{1}(I_{A^{\ast}})\leq\alpha.
\]
When $\rho_{1}(X^{\ast})>0$, we have
\[
\rho_{1}(X^{N})\leq\rho_{1}((X^{\ast}-\frac{1}{N})^{+}I_{A^{c}})+\rho
_{1}(I_{A})\leq\rho_{1}((X^{\ast}-\frac{1}{N})^{+})+\rho_{1}(I_{A^{\ast\ast}%
})\leq\rho_{1}(X^{\ast})\leq\alpha.
\]
Thus, $X^{N}$ belongs to $\mathcal{X}_{\alpha}$. On the other hand,
\[%
\begin{array}
[c]{r@{}l}%
E_{Q}[X^{N}]= & E_{\lambda Q^{c}}[X^{N}]+E_{(1-\lambda)Q^{p}}[X^{N}]\\
\geq & E_{\lambda Q^{c}}[X^{\ast}]-\frac{1}{N}+E_{(1-\lambda)Q^{p}}[I_{A}]\\
= & E_{\lambda Q^{c}}[X^{\ast}]-\frac{1}{N}+1-\lambda\\
> & E_{\lambda Q^{c}}[X^{\ast}]+\lambda_{0}=E_{Q}[X^{\ast}],
\end{array}
\]
which conflicts with Lemma \ref{lem4.2}. Thus, $E_{(1-\lambda)Q^{p}}[X^{\ast
}]=1-\lambda$.
\end{proof}

\begin{proposition}
\label{lem4.4} Under (H1), if $X^{\ast}$ is an optimal test of Problem
\ref{problem2.1}, then for any $Q\in\mathcal{\bar{Q}}$, we have
\[
E_{Q}[X^{\ast}]=\sup_{X\in\mathcal{X}_{\alpha}}E_{\lambda Q^{c}}[X]+\sup
_{X\in\mathcal{X}_{\alpha}}E_{(1-\lambda)Q^{p}}[X].
\]

\end{proposition}

\begin{proof}
It is easy to verify that
\[
\sup_{X\in\mathcal{X}_{\alpha}}E_{Q}[X]\leq\sup_{X\in\mathcal{X}_{\alpha}%
}E_{\lambda Q^{c}}[X]+\sup_{X\in\mathcal{X}_{\alpha}}E_{(1-\lambda)Q^{p}}[X].
\]
So we just need to prove the converse inequality.

The case of $\lambda=1$ ($Q$ itself is a probability measure) is trivial. In
the following, we always assume $\lambda<1$. Set
\[
\gamma^{c}=\sup\limits_{X\in\mathcal{X}_{\alpha}}E_{\lambda Q^{c}}[X].
\]
There exists a sequence $\{X_{n}^{c}\}_{n\in\mathbb{N}}\subset\mathcal{X}%
_{\alpha}$ such that
\[
E_{\lambda Q^{c}}[X_{n}^{c}]\geq\gamma^{c}-\frac{1}{n}.
\]

For any $X_{n}^{c}\in\mathcal{X}_{\alpha}$, using exactly the same technique
and as in Proposition \ref{lem4.3}, we can construct a test%
\[
X_{n}^{N}=(X_{n}^{c}-\frac{1}{N})^{+}I_{A^{c}}+I_{A}%
\]
which belongs to $\mathcal{X}_{\alpha}$. Then, we obtain
\[%
\begin{array}
[c]{r@{}l}%
E_{Q}[X_{n}^{N}] & =E_{\lambda Q^{c}}[X_{n}^{N}]+E_{(1-\lambda)Q^{p}}%
[X_{n}^{N}]\\
& \geq E_{\lambda Q^{c}}[X_{n}^{c}]-\frac{\lambda}{N}+E_{(1-\lambda)Q^{p}%
}[I_{A}]\\
& =E_{\lambda Q^{c}}[X_{n}^{c}]-\frac{\lambda}{N}+1-\lambda\\
& \geq\gamma^{c}+1-\lambda-(\frac{\lambda}{N}+\frac{1}{n}).
\end{array}
\]
Since $N$ and $n$ are arbitrary,
\[
\sup_{X\in\mathcal{X}_{\alpha}}E_{Q}[X]\geq E_{Q}[X_{n}^{N}]\geq\gamma
^{c}+1-\lambda.
\]

On the other hand, $\sup\limits_{X\in\mathcal{X}_{\alpha}}E_{(1-\lambda)Q^{p}%
}[X]\leq1-\lambda$. Thus,
\[
\sup_{X\in\mathcal{X}_{\alpha}}E_{Q}[X]\geq\sup_{X\in\mathcal{X}_{\alpha}%
}E_{\lambda Q^{c}}[X]+\sup_{X\in\mathcal{X}_{\alpha}}E_{(1-\lambda)Q^{p}}[X].
\]
This completes the proof.
\end{proof}

\begin{theorem}
\label{the4.1} Under (H1), if $X^{\ast}$ is an optimal test of Problem
\ref{problem2.1}, then for any $Q\in\mathcal{\bar{Q}}$, it is an optimal test
of the following problem:
\begin{equation}
\sup_{X\in\mathcal{X}_{\alpha}}E_{\lambda Q^{c}}[X]. \label{0.3}%
\end{equation}

\end{theorem}

\begin{proof}
By Proposition \ref{lem4.4},
\[
E_{\lambda Q^{c}}[X^{\ast}]+E_{(1-\lambda)Q^{p}}[X^{\ast}]=E_{Q}[X^{\ast
}]=\sup_{X\in\mathcal{X}_{\alpha}}E_{\lambda Q^{c}}[X]+\sup_{X\in
\mathcal{X}_{\alpha}}E_{(1-\lambda)Q^{p}}[X].
\]
Since
\[
E_{\lambda Q^{c}}[X^{\ast}]\leq\sup_{X\in\mathcal{X}_{\alpha}}E_{\lambda
Q^{c}}[X]\quad\text{and}\quad E_{(1-\lambda)Q^{p}}[X^{\ast}]\leq\sup
_{X\in\mathcal{X}_{\alpha}}E_{(1-\lambda)Q^{p}}[X],
\]
we have
\[
E_{\lambda Q^{c}}[X^{\ast}]=\sup_{X\in\mathcal{X}_{\alpha}}E_{\lambda Q^{c}%
}[X].
\]
This completes the proof.
\end{proof}

If $\lambda=0$, i.e., $Q$ is a pure additive set function, then the problem
(\ref{0.3}) is meaningless. In order to avoid this case, we need the following hypothesis:

(H2) For any sequence $\{A_{n}\}_{n\in\mathbb{N}}$ such that $A_{n}%
\downarrow\emptyset$, we have $\lim\limits_{n\rightarrow\infty}\rho
_{1}(I_{A_{n}})<1$ and $\lim\limits_{n\rightarrow\infty}\rho_{2}(I_{A_{n}})<1$.

\begin{corollary}
\label{cor3.12} Under (H1) and (H2), if $X^{\ast}$ is an optimal test of
Problem \ref{problem2.1}, then for any $Q\in\mathcal{\bar{Q}}$, it is also an
optimal test of the following problem:
\begin{equation}
\sup_{X\in\mathcal{X}_{\alpha}}E_{Q^{c}}[X]. \label{4.3}%
\end{equation}

\end{corollary}

\begin{proof}
Note that (H2) guarantees $\lambda>0$. We can get our result directly from
Theorem \ref{the4.1}.
\end{proof}

By Corollary \ref{cor3.12}, as long as we prove that all optimal tests of
(\ref{4.3}) have a specific form, then $X^{\ast}$ will have the same form.

\subsection{Representation of optimal tests}

In this subsection, we will assume (H1) and (H2) hold and focus on solving the
problem of (\ref{4.3}). Without causing confusion, we still use $\gamma^{c}$
to denote $\sup\limits_{X\in\mathcal{X}_{\alpha}}E_{Q^{c}}[X]$. Then, by
Corollary \ref{cor3.12}, $E_{Q^{c}}[X^{\ast}]=\gamma^{c}$.

Firstly, We consider that for any optimal test $X^{\ast}$ of (\ref{4.3}), we
have $\rho_{1}(X^{\ast})=\alpha$ or equivalently for all $X\in\mathcal{X}%
_{\alpha}$ such that $\rho_{1}(X)<\alpha$, we have $E_{Q^{c}}[X]<\gamma^{c}$.
In this case, the following result holds.

\begin{proposition}
\label{lem4.6} Assume that for all $X\in X^{\ast}$ such that $\rho
_{1}(X)<\alpha$, we have $E_{Q^{c}}[X]<\gamma^{c}$. If $X^{\ast}$ is an
optimal test of (\ref{4.3}), then $X^{\ast}$ is also an optimal test of the
following problem:
\[
\inf_{Y\in\mathcal{Y}_{\alpha}}\rho_{1}(Y),
\]
where $\mathcal{Y}_{\alpha}=\{Y\mid E_{Q^{c}}[Y]\geq\gamma^{c},Y\in
\lbrack0,1],Y\in\mathcal{X}\}$.
\end{proposition}

\begin{proof}
Note that $E_{Q^{c}}[Y]<\gamma^{c}$ when $\rho_{1}(Y)<\alpha$. Then, for any
$Y\in\mathcal{Y}_{\alpha}$, we have $\rho_{1}(Y)\geq\alpha$. With $X^{\ast}%
\in\mathcal{Y}_{\alpha}$ and $\rho_{1}(X^{\ast})=\alpha$, the result is proved.
\end{proof}

The other case is that there exists an optimal test $\bar{X}^{\ast}$ of
(\ref{4.3}) such that $\rho_{1}(\bar{X}^{\ast})<\alpha$. In this case, we can
show for any optimal test function $X^{\ast}$ of (\ref{4.3}), we have
$Q^{c}(\{X^{\ast}\not =1\})=0$. Two lemmas are needed before we prove this conclusion.

\begin{lemma}
\label{pro4.1} If there exists an optimal test $\bar{X}^{\ast}$ of (\ref{4.3})
such that $\rho_{1}(\bar{X}^{\ast})<\alpha$, then we have
\[
E_{Q^{c}}[I_{\{\bar{X}^{\ast}\not =1\}}]=0.
\]

\end{lemma}

\begin{proof}
Set
\[
\hat{X}^{\ast}=(\bar{X}^{\ast}+\alpha-\rho_{1}(\bar{X}^{\ast}))\wedge1.
\]
It is obvious that $\rho_{1}(\hat{X}^{\ast})\leq\alpha$ and $E_{Q^{c}}[\hat
{X}^{\ast}]\geq E_{Q^{c}}[\bar{X}^{\ast}]$. $\rho_{1}(\hat{X}^{\ast}%
)\leq\alpha$ implies that $\hat{X}^{\ast}\in\mathcal{X}_{\alpha}$ and
$E_{Q^{c}}[\hat{X}^{\ast}]=\gamma^{c}$. Then, we have
\[
E_{Q^{c}}[\hat{X}^{\ast}-\bar{X}^{\ast}]=0.
\]
By the monotone convergence theorem,
\[
E_{Q^{c}}[I_{\{\bar{X}^{\ast}\not =1\}}]=\lim_{n\rightarrow\infty}E_{Q^{c}%
}[n(\hat{X}^{\ast}-\bar{X}^{\ast})\bigwedge1]\leq\lim_{n\rightarrow\infty
}E_{Q^{c}}[n(\hat{X}^{\ast}-\bar{X}^{\ast})]=0.
\]
This completes the proof.
\end{proof}

\begin{lemma}
\label{cor4.3} If there exists an optimal test function $\bar{X}^{\ast}$ of
(\ref{4.3}) satisfying $\rho_{1}(\bar{X}^{\ast})<\alpha$, then $\gamma^{c}=1$.
\end{lemma}

\begin{proof}
Since $\rho_{1}(\bar{X}^{\ast})<\alpha$, by Lemma \ref{pro4.1}, the set
$\{\bar{X}^{\ast}\not =1\}$ satisfies
\[
E_{Q^{c}}[I_{\{\bar{X}^{\ast}\not =1\}}]=0.
\]
Thus,
\[
\gamma^{c}=E_{Q^{c}}[\bar{X}^{\ast}]=E_{Q^{c}}[I_{\{\bar{X}^{\ast}=1\}}]=1.
\]
This completes the proof.
\end{proof}

\begin{proposition}
\label{cor4.4} If there exists an optimal test function $\bar{X}^{\ast}$ of
(\ref{4.3}) satisfying $\rho_{1}(\bar{X}^{\ast})<\alpha$, then for any optimal
test function $X^{\ast}$ of (\ref{4.3}), we have $Q^{c}(\{X^{\ast}%
\not =1\})=0$.
\end{proposition}

\begin{proof}
For any optimal test function $X^{\ast}$ of (\ref{4.3}), by Lemma
\ref{cor4.3}, $E_{Q^{c}}[X^{\ast}]=\gamma^{c}=1$. Since $1-X^{\ast}\geq0$ and
$E_{Q^{c}}[1-X^{\ast}]=0$, we deduce that $Q^{c}(\{X^{\ast}\not =1\})=0$.
\end{proof}

Now, we give our main result.

\begin{theorem}
\label{main-result} For any optimal test $X^{\ast}$ of (\ref{4.3}), there
exists a probability measure $P^{c}$ such that
\[
X^{\ast}=I_{\{\kappa H_{Q^{c}}>G_{P^{c}}\}}+BI_{\{\kappa H_{Q^{c}}=G_{P^{c}%
}\}},\quad K-a.s.,
\]
where $H_{Q^{c}}$ and $G_{P^{c}}$ are the Radon-Nikodym derivatives of $Q^{c}$
and $P^{c}$ with respect to $K:=\frac{P^{c}+Q^{c}}{2}$, $\kappa\in
\mathbb{R}\cup\{+\infty\}$ and $B$ is a random variable with values in the
interval [0, 1].
\end{theorem}

\begin{proof}
We divide into two cases to prove this result.

The first case, for all $X\in\mathcal{X}_{\alpha}$ such that $\rho
_{1}(X)<\alpha$, we have $E_{Q^{c}}[X]<\gamma^{c}$. By Proposition
\ref{lem4.6}, $X^{\ast}$ is an optimal test of the following problem:
\begin{equation}
\inf_{Y\in\mathcal{Y}_{\alpha}}\rho_{1}(Y).\label{4.4}%
\end{equation}

Using the same analysis as in subsection \ref{property}, there exists a
finitely additive set function $P\in\mathcal{P}$ such that for any optimal
test of (\ref{4.4}),
\[
E_{P}[Y^{\ast}]=\inf_{Y\in\mathcal{Y}_{\alpha}}E_{P}[Y]=\inf_{Y\in
\mathcal{Y}_{\alpha}}\rho_{1}(Y)
\]
and $P$ has the unique decomposition
\[
P=\tau P^{c}+(1-\tau)P^{p},
\]
where $\tau\in\lbrack0,1]$, $P^{c}$ is a probability measure and $P^{p}$ is a
pure additive set function.

Since $E_{Q^{c}}$ plays the same role as $\rho_{1}$ as in subsection
\ref{property} and $Q^{c}$ is a probability measure, Assumption (H1) holds.
Under Assumption (H2), we have $\tau\in(0,1]$. Then, similarly analysis as in
Corollary \ref{cor3.12}, we have the countably part of $P$ satisfying
\begin{equation}
E_{P^{c}}[Y^{\ast}]=\inf_{Y\in\mathcal{Y}_{\alpha}}E_{P^{c}}[Y]. \label{4.5}%
\end{equation}

Since $P^{c}$ and $Q^{c}$ are both probability measures, by the classical
Neyman-Pearson lemma (see \cite{CK}), any optimal test $Y^{\ast}$ of
(\ref{4.5}) has the following form:
\[
Y^{\ast}=I_{\{\kappa H_{Q^{c}}>G_{P^{c}}\}}+B\cdot I_{\{\kappa H_{Q^{c}%
}=G_{P^{c}}\}},\quad K-a.s.,
\]
where
\[
K=\frac{P^{c}+Q^{c}}{2},
\]%
\[
\kappa=\inf\{u\geq0\mid Q^{c}(uH_{Q^{c}}\geq G_{P^{c}})\geq\gamma^{c}\}
\]
and $B$ is a random variable taking values in the interval [0, 1].

Since all the optimal test functions of (\ref{4.5}) have the above form, then
any optimal test function of (\ref{4.4}) has the same form. So does the
optimal test function of (\ref{4.3}).

The second case, there exists an optimal test $\bar{X}^{\ast}$ of (\ref{4.3})
such that $\rho_{1}(\bar{X}^{\ast})<\alpha$. By Proposition \ref{cor4.4}, for
any optimal test $X^{*}$ of (\ref{4.3}), we have $Q^{c}(\{X^{\ast}%
\not =1\})=0$. Then, $\{X^{\ast}\not =1\}\subset\{H_{Q^{c}}=0\}$ and
$\{H_{Q^{c}}>0\}\subset\{X^{\ast}=1\}$. Thus, $X^{\ast}$ can be expressed as
\[
X^{\ast}=I_{\{H_{Q^{c}}>0\}}+BI_{\{H_{Q^{c}}=0\}},\quad K-a.s.,
\]
where $B$ is a random variable taking values in $[0,1]$. It can also be
written as the form in the first case, just $\kappa=+\infty$ here.

Combining these two cases, we can get our result..
\end{proof}

By Theorem \ref{main-result}, we have proved that for any optimal test
$X^{\ast}$ of (\ref{4.3}), it must have the reminiscent form as in classical
case. Since any optimal test of Problem \ref{problem2.1} is also the optimal
test of (\ref{4.3}), then the optimal test of the initial Problem
\ref{problem2.1} has the reminiscent form as in classical case.

Next, some examples will be given. Example \ref{exa5.3} shows the obtained
result is only a necessary condition of the optimal test.

\begin{example}
\label{exa5.3} Let $\Omega=[0,1]$ and $\mathcal{F}$ be the collection of all
the Borel set on $[0,1]$. $\mathcal{P}:=\{P\}$, where
\[
P(\omega)=\Bigg\{%
\begin{array}
[c]{l@{}c}%
\frac{1}{2},\quad & \quad\omega=\frac{1}{2};\\
\frac{1}{2},\quad & \quad\omega=1
\end{array}
\]
and $\mathcal{Q}:=\{\delta_{x}:x\in\lbrack0,1)\}$, where $\delta_{x}$ is the
Dirac measure.

If $\alpha=\frac{1}{2}$, then the optimal test is
\[
X^{\ast}=\Bigg\{%
\begin{array}
[c]{l@{}c}%
1,\quad & \quad\omega\in\lbrack0,1);\\
0,\quad & \quad\omega=1.
\end{array}
\]
and it is unique.

Every $\delta_{x}\in\mathcal{Q}$ can be chosen as $Q$, but there does not
exist a $Q$ such that $\{H_{Q}>\kappa G_{P}\}=A$ for some $\kappa$, where
$A=\{\omega:X^{\ast}=1\}$, $H_{Q}$ and $G_{P}$ are Radon-Nikodym derivatives
with respect to $K=\frac{Q+P}{2}$.
\end{example}

Example \ref{exa5.4} shows the choice of $Q$ impacts on finding the optimal test.

\begin{example}
\label{exa5.4} Let $\Omega=\{\omega_{1},\omega_{2},\omega_{3}\}$,
$\mathcal{F}$ be all possible combinations of elements in $\Omega$.
$\mathcal{P}=\{P\}$ and $\mathcal{Q}=\{Q_{1},Q_{2}\}$, where%
\[
P=\left\{
\begin{array}
[c]{l@{}c}%
\frac{1}{4},\quad & \quad\omega=\omega_{1};\\
\frac{1}{4},\quad & \quad\omega=\omega_{2};\\
\frac{1}{2},\quad & \quad\omega=\omega_{3},
\end{array}
\right.  \quad Q_{1}=\left\{
\begin{array}
[c]{l@{}c}%
\frac{1}{2},\quad & \quad\omega=\omega_{1};\\
\frac{1}{2},\quad & \quad\omega=\omega_{2};\\
0,\quad & \quad\omega=\omega_{3}%
\end{array}
\right.  \,and\quad Q_{2}=\left\{
\begin{array}
[c]{l@{}c}%
1,\quad & \quad\omega=\omega_{1};\\
0,\quad & \quad\omega=\omega_{2};\\
0,\quad & \quad\omega=\omega_{3}.
\end{array}
\right.
\]

If $\alpha=\frac{1}{2}$, it is obvious that the optimal test is $X^{\ast
}=I_{\{\omega_{1}\}}+I_{\{\omega_{2}\}}$. Furthermore, both $Q_{1}$ and
$Q_{2}$ can be considered as $Q$. If we choose $Q_{2}$ as $Q$, $I_{\{\omega
_{1}\}}$ satisfy
\[
\rho_{1}[I_{\{\omega_{1}\}}]=\frac{1}{4}<\frac{1}{2},E_{Q_{2}}[I_{\{\omega
_{1}\}}]=1=\sup_{X\in\mathcal{X}_{\alpha}}E_{Q_{2}}[X],
\]
while it is not the optimal test.
\end{example}

Example \ref{exa5.1} shows the second case in the proof of Theorem
\ref{main-result} does exist.

\begin{example}
\label{exa5.1} Let $\Omega=[0,1]$ and $\mathcal{F}$ be all the Borel set on
$[0,1]$. $\mathcal{P}=\{\delta_{0}\}$ and $\mathcal{Q}=\{\delta_{1}\}$, where
$\delta_{0}$ and $\delta_{1}$ are Dirac measures, i.e.,%

\[
\delta_{0}=\left\{
\begin{array}
[c]{l@{}c}%
1,\quad & \quad\omega=0;\\
0,\quad & \quad\text{otherwise},
\end{array}
\right.  \quad\quad\delta_{1}=\left\{
\begin{array}
[c]{l@{}c}%
1,\quad & \quad\omega=1;\\
0,\quad & \quad\text{otherwise}.
\end{array}
\right.
\]

For any given $0<\alpha<1$, we have indicator function $I_{\{1\}}$ is always
the optimal test function and it obviously has $0-1$ structure while $\rho
_{1}(I_{\{1\}})=E_{P}[I_{\{1\}}]=0<\alpha$. Its representation form is out of
the framework in \cite{CK}.
\end{example}

In the end of this subsection, we give a necessary and sufficient condition
for judging whether there exists an optimal test function $\bar{X}^{\ast}$
such that $\rho_{1}(\bar{X}^{\ast})<\alpha$.

\begin{proposition}
\label{the4.3} Denote
\[
\beta=\sup_{B\in\mathcal{B}}\bar{\rho}_{1}(I_{B}),
\]
where $\mathcal{B}=\{B\in\mathcal{F}\mid\bar{\rho}_{1}(I_{B})>0,E_{Q^{c}%
}(I_{B})=0\}$. If $\mathcal{B}$ is empty, we define $\beta=0$. Then, for any
$\alpha\in(0,1)$, there exists $\bar{X}^{\ast}$ such that $\rho_{1}(\bar
{X}^{\ast})<\alpha$ and $E_{Q^{c}}[\bar{X}^{\ast}]=\gamma^{c}$ if and only if
$\beta>1-\alpha$.
\end{proposition}

\begin{proof}
$\Leftarrow$: If $\beta>1-\alpha$, there exists a set $\widehat{B}%
\in\mathcal{F}$ such that $\bar{\rho}_{1}(I_{\widehat{B}})>1-\alpha$ and
$E_{Q^{c}}(I_{\widehat{B}})=0$. Then, $\rho_{1}(I_{\widehat{B}^{c}})<\alpha$
and $E_{Q^{c}}[I_{\widehat{B}^{c}}]=1$, i.e., $I_{\widehat{B}^{c}}$ is an
optimal test of (\ref{4.3}) satisfying $\rho_{1}(I_{\widehat{B}^{c}})<\alpha$.

$\Rightarrow$: If there exists $\bar{X}^{\ast}$ such that $\rho_{1}(\bar
{X}^{\ast})<\alpha$ and $E_{Q^{c}}[\bar{X}^{\ast}]=\gamma^{c}$, from Corollary
\ref{cor4.3}, we know $\gamma^{c}=1$. Since $I_{\{\bar{X}^{\ast}=1\}}\leq
\bar{X}^{\ast}$, we have $\rho_{1}(I_{\{\bar{X}^{\ast}=1\}})<\alpha$, i.e.,
$\bar{\rho}_{1}[I_{\{\bar{X}^{\ast}\not =1\}}]>1-\alpha$. Since $1-\bar
{X}^{\ast}>0$ on set $\{\bar{X}^{\ast}\not =1\}$ and $E_{Q^{c}}[(1-\bar
{X}^{\ast})I_{\{\bar{X}^{\ast}\not =1\}}]=E_{Q^{c}}[1-\bar{X}^{\ast}]=0$, we
deduce that $E_{Q^{c}}[I_{\{\bar{X}^{\ast}\not =1\}}]=0$. Thus, $\beta\geq
\rho_{1}[I_{\{\bar{X}^{\ast}\not =1\}}]>1-\alpha$.
\end{proof}

\section{Results on some special sublinear expectations}

In this section, we show some results when the sublinear expectation $\rho$
can be represented by a family of probability measures.

\subsection{{Existence of optimal tests}}

In this section, we will give a sufficient condition for the existence of the
optimal test of Problem \ref{problem2.1}.

\begin{lemma}
\label{lem-fatou lemma} If a sublinear expectation $\rho$ can be represented
by a family of probability measures $\mathcal{M}$, i.e., $\rho
(X)=\underset{P\in\mathcal{M}}{\sup}E_{P}[X],$ then for any sequence
$\{X_{n}\}_{n\in\mathbb{N}}$ satisfying there exists a constant $M\in
\mathbb{R}$ such that for any $n\in\mathbb{N}$, $|X_{n}|\leq M$, we have
\[
\rho(\mathop{\lim\inf}_{n}X_{n})\leq\mathop{\lim\inf}_{n}\rho(X_{n})
\]
and
\[
\bar{\rho}(\mathop{\lim\sup}_{n}X_{n})\geq\mathop{\lim\sup}_{n}\bar{\rho
}(X_{n})
\]
where $\bar{\rho}$ is the conjugate operator of $\rho$.
\end{lemma}

\begin{proof}
We only prove the first inequality. Let $\zeta_{n}=\inf_{k\geq n}X_{k}$. Then,
$\zeta_{n}\leq X_{n}$ and $\{\zeta_{n}\}_{n\in\mathbb{N}}$ is an increasing
sequence. Thus,
\[
\rho(\mathop{\lim\inf}_{n}X_{n})=\rho(\lim_{n}\zeta_{n})=\lim_{n}\rho
(\zeta_{n})\leq\mathop{\lim\inf}_{n}\rho(X_{n}).
\]
This completes the proof.
\end{proof}

\begin{theorem}
Suppose that there exists a probability measure $K$ and two families of
probability measures $\mathcal{P}$ and $\mathcal{Q}$ such that for any
$P\in\mathcal{P}$, $Q\in\mathcal{Q}$, we have $P\ll K$, $Q\ll K$. If for any
$X\in\mathcal{X}$, we have
\[
\rho_{1}(X)=\sup\limits_{P\in\mathcal{P}}E_{P}[X]\quad\text{and}\quad\rho
_{2}(X)=\sup\limits_{Q\in\mathcal{Q}}E_{Q}[X],
\]
then the optimal test of Problem \ref{problem2.1} exists.
\end{theorem}

\begin{proof}
Take a sequence $\{X_{n}\}_{n\in\mathbb{N}}\subset\mathcal{X}_{\alpha}$ such
that
\[
\bar{\rho}_{2}(X_{n})\geq\gamma_{\alpha}-\dfrac{1}{2^{n}},
\]
where $\gamma_{\alpha}=\sup\limits_{X\in\mathcal{X}_{\alpha}}\bar{\rho}%
_{2}(X)$.

By Koml\'{o}s theorem, there exists a subsequence $\{X_{n_{i}}\}_{i\geq
1}\subset\{X_{n}\}_{n\in\mathbb{N}}$ and a random variable $\hat{X}$ such
that
\[
\lim_{k\rightarrow\infty}\frac{1}{k}\sum_{i=1}^{k}X_{n_{i}}=\hat{X},\quad
K-a.s..
\]
Since $0\leq\mid X_{n}\mid\leq1$ for any $n\in\mathbb{N}$, $\hat{X}$ takes
values in $[0,1]$. By Lemma \ref{lem-fatou lemma},
\[
\rho_{1}(\hat{X})\leq\liminf_{k\rightarrow\infty}\frac{1}{k}\sum_{i=1}^{k}%
\rho_{1}(X_{n_{i}})\leq\alpha
\]
and
\[
\bar{\rho}_{2}(\hat{X})\geq\limsup_{k\rightarrow\infty}\frac{1}{k}\sum
_{i=1}^{k}\bar{\rho}_{2}(X_{n_{i}})\geq\lim_{k\rightarrow\infty}%
(\gamma_{\alpha}-\frac{1}{k})=\gamma_{\alpha}.
\]
Thus, $\hat{X}$ is an optimal test of Problem \ref{problem2.1}.
\end{proof}

\subsection{Test functions in $L_{c}^{1}$-space}

Some theories consider a small space instead of the whole bounded measurable
function space. For an example, the $G$-expectation theory introduced by Peng
\cite{Peng} considers the $L_{c}^{1}$-space. In this section, we will study
the hypothesis test when test functions are selected in $L_{c}^{1}$-space.

Firstly, we give the definition of the $L_{c}^{1}$-space, which comes from
\cite{DHP}.

Let $\Omega$ be a complete separable metric space equipped with the distance
$d$, $\mathcal{B}(\Omega)$ be the Borel $\sigma$-field of $\Omega$ and
$C_{b}(\Omega)$ be all continuous bounded $\mathcal{B}(\Omega)$-measurable
real functions. For two sublinear expectations $\rho_{1}$ and $\rho_{2}$, take
$\rho(X)=\rho_{1}(X)\bigvee\rho_{2}(X)$. Then, $\rho$ is a sublinear
expectation. Let
\[
c(A)=\rho(I_{A}).
\]

\begin{definition}
\label{def6.1} The set $A$ is polar if $c(A)=0$ and we say a property holds
"quasi-surely" (q.s.) if it holds outside a polar set.
\end{definition}

\begin{definition}
With the norm $||X||_{L^{1}}=\rho(|X|)$, $L_{c}^{1}$ space is the completeness
of the $C_{b}(\Omega)$ under the $L^{1}$-norm.
\end{definition}

For a given significance level $\alpha\in(0,1)$, set%
\[
\mathcal{\hat{X}}_{\alpha}=\{X\in L_{c}^{1};0\leq X\leq1,\rho_{1}(X)\leq
\alpha\}.
\]
The hypothesis testing problem is:

\begin{problem}
\label{problem for continuous} For two sublinear expectations $\rho_{1}$ and
$\rho_{2}$, whether there exists an $X^{\ast}\in\mathcal{\hat{X}}_{\alpha}$
such that
\begin{equation}
\bar{\rho}_{2}(X^{\ast})=\sup_{X\in\mathcal{\hat{X}}_{\alpha}}\bar{\rho}%
_{2}(X).
\end{equation}

\end{problem}

\begin{definition}
\label{def-fatou property for continuous} We say sublinear expectation $\rho$
is continuous from above on $L_{c}^{1}$ if for any sequence $\{f_{n}%
\}_{n\in\mathbb{N}}\subset L_{c}^{1}$ such that $f_{n}\downarrow0$, $c$-q.s.,
we have
\[
\rho(f_{n})\downarrow0.
\]

\end{definition}

By Corollary 33 in \cite{DHP}, a sublinear expectation $\rho$ is continuous
from above on $L_{c}^{1}$ if and only if there exists a weakly compact
probability measure set $\mathcal{M}$ such that $\rho(X)=\sup_{P\in
\mathcal{M}}E_{P}[X]$ for any $X\in L_{c}^{1}$.

\begin{theorem}
If sublinear expectations $\rho_{1}$ and $\rho_{2}$ are both continuous from
above on $L_{c}^{1}$ and the optimal test of Problem
\ref{problem for continuous} exists, then any optimal test $X^{\ast}$ of
Problem \ref{problem for continuous} has the form:
\[
X^{\ast}=I_{\{\kappa H_{Q}>G_{P}\}}+BI_{\{\kappa H_{Q}=G_{P}\}},\quad K-a.s.,
\]
where $Q$ and $P$ are defined as in section \ref{sec3}, $K=\frac{P+Q}{2}$,
$\kappa\in\mathbb{R}\cup\{+\infty\}$ and $B$ is a random variable with values
in the interval $[0,1]$.
\end{theorem}

\begin{proof}
The whole proof is similar as in section \ref{sec3}. Since $\rho_{1}$ and
$\rho_{2}$ are both continuous from above on $L_{c}^{1}$, $P$ and $Q$ chosen
as in section \ref{sec3} satisfy that for any $\{f_{n}\}_{n\in\mathbb{N}%
}\subset L_{c}^{1}$ satisfying $f_{n}\downarrow0$, $c$-q.s., we have
$E_{P}(f_{n})\downarrow0$ and $E_{Q}(f_{n})\downarrow0$. By the Daniell-Stone
theorem (see p. 59 of \cite{DM}), $P$ and $Q$ can be chosen both as
probability measures. Thus, the two Assumptions (H1) and (H2) in section
\ref{sec3} hold naturally. We omit the rest of the proof.
\end{proof}

Peng introduced $G$-Expectation which is a specific sublinear expectation in
\cite{Peng}. In \cite{DHP}, Denis Hu and Peng gave a specific represent form
for $G$-expectation. Under the framework of \cite{DHP} and \cite{Peng}, we
introduce the following example for our hypothesis testing problem.

\begin{example}
\label{exa5.5} Let $(\Omega, \mathcal{F}, P_{0})$ be a probability space and
$(W_{t})_{t\geq0}$ be the 1-dimensional Brownian motion in this space. The
filtration $\mathbb{F}=\{ \mathcal{F}_{t}\}_{0\leq t\leq T}$ is the augmented
$\sigma$-algebra generated by $W(\cdot)$. Let $\Theta$ be a given bounded and
closed subset in $\mathbb{R}$. Denote $\mathcal{A}^{\Theta}_{t, T}$ as the
collection of all $\Theta$-valued $\mathbb{F}$-adapted processes on an
interval $[t, T]\subset[0, \infty)$. For each fixed $\theta\in\mathcal{A}%
^{\Theta}_{t, T}$, denote
\[
B^{0, \theta}_{t}:=\int^{t}_{0}\theta_{s}dW_{s}.
\]

Let $P_{\theta}:=P_{0}\circ(B_{\cdot}^{0,\theta})^{-1}$. Then, the
G-expectation $\mathbb{E}[\cdot]$ introduced by Peng can be written as
\[
\mathbb{E}[\phi(B_{t_{1}}^{0},B_{t_{2}}^{t_{1}},\cdots,B_{t_{n}}^{t_{n-1}%
})]=\sup_{\theta\in\mathcal{A}_{0,T}^{\Theta}}E_{P_{\theta}}[\phi(B_{t_{1}%
}^{0},B_{t_{2}}^{t_{1}},\cdots,B_{t_{n}}^{t_{n-1}})]
\]

Given two families of probability measures $\mathcal{P}=\{P_{\theta},\theta
\in\Theta_{1}\}$ and $\mathcal{Q}=\{Q_{\theta},\theta\in\Theta_{2}\}$,
$\Theta_{1}\cap\Theta_{2}=\Phi$. It is easy to check that $\rho_{1}%
(I_{\{\langle B\rangle_{1}\in\Theta_{2}\}})=0$ and $\rho_{2}(I_{\{\langle
B\rangle_{1}\in\Theta_{2}\}})=1$, where $\langle B\rangle$ is the quadratic
variation process of $B$. If we want to discriminate the G-expectations
$\rho_{1}(\cdot)$ and $\rho_{2}(\cdot)$, i.e., to discriminate $\Theta_{1}$
and $\Theta_{2}$, for any significance level $\alpha$, then $I_{\{\langle
B\rangle_{1}\in\Theta_{2}\}}$ is an optimal test of Problem \ref{problem2.1}.
\end{example}

\end{document}